\documentclass[10pt,a4paper]{article}

\usepackage{amsmath}
\usepackage{amssymb}
\usepackage{ntheorem} 
\usepackage{hyperref} 
\usepackage{verbatim}
\usepackage{authblk}
\usepackage{xcolor}
\usepackage{enumitem}
\usepackage{cite}

\allowdisplaybreaks

\theorembodyfont{\slshape}

\newtheorem*{thmn}{Theorem}

\newtheorem{thm}{Theorem}[section]
\newtheorem{cor}[thm]{Corollary}
\newtheorem{lem}[thm]{Lemma}
\newtheorem{prop}[thm]{Proposition}
\theoremstyle{definition}
\newtheorem{defn}[thm]{Definition}
\newtheorem{defnr}[thm]{Definition \& Remark}
\newtheorem{ex}[thm]{Example}
\newtheorem{rem}[thm]{Remark}

\newtheorem*{nota}{Notations}

\numberwithin{equation}{section}

\newenvironment{proof}{\noindent\textit{Proof. }}{\hfill $\Box$\linebreak}

\newcommand*{\Chi}{\protect\raisebox{2.5pt}{$\chi$}}

\DeclareMathSymbol{\ordinaryl}{\mathalpha}{letters}{`l}     
\begingroup                             
\lccode`\~=`\l
\lowercase{\gdef ~{\ifnum\the\mathgroup=-1 \ell \else \ordinaryl \fi}}
\endgroup
\mathcode `\l="8000

\newcommand{\bdefn}{\begin{defn}}
\newcommand{\edefn}{\end{defn}}
\newcommand{\bdefnr}{\begin{defnr}}
\newcommand{\edefnr}{\end{defnr}}
\newcommand{\benum}{\begin{enumerate}[label=(\arabic*),leftmargin=1.8em]}
\newcommand{\eenum}{\end{enumerate}}
\newcommand{\bthmn}{\begin{thmn}}
\newcommand{\ethmn}{\end{thmn}}
\newcommand{\bpropn}{\begin{propn}}
\newcommand{\epropn}{\end{propn}}
\newcommand{\bthm}{\begin{thm}}
\newcommand{\ethm}{\end{thm}}
\newcommand{\bnota}{\begin{nota}}
\newcommand{\enota}{\end{nota}}
\newcommand{\bproof}{\begin{proof}}
\newcommand{\eproof}{\end{proof}}
\newcommand{\bprop}{\begin{prop}}
\newcommand{\eprop}{\end{prop}}
\newcommand{\bcor}{\begin{cor}}
\newcommand{\ecor}{\end{cor}}
\newcommand{\blem}{\begin{lem}}
\newcommand{\elem}{\end{lem}}
\newcommand{\brem}{\begin{rem}}
\newcommand{\erem}{\end{rem}}
\newcommand{\bex}{\begin{ex}}
\newcommand{\eex}{\end{ex}}
\newcommand{\boq}{\begin{openq}}
\newcommand{\eoq}{\end{openq}}
\newcommand{\btab}{\begin{tabular}}
\newcommand{\etab}{\end{tabular}}
\newcommand{\bctab}{\begin{center}\begin{tabular}}
\newcommand{\ectab}{\end{tabular}\end{center}}

\newcommand{\ba}{\begin{array}}
\newcommand{\ea}{\end{array}}
\newcommand{\bea}{\begin{eqnarray}}
\newcommand{\eea}{\end{eqnarray}}
\newcommand{\bean}{\begin{eqnarray*}}
\newcommand{\eean}{\end{eqnarray*}}

\newcommand{\cbb}{\mathbb{C}}

\newcommand{\kbb}{\mathbb{K}}

\newcommand{\nbb}{\mathbb{N}}

\newcommand{\rbb}{\mathbb{R}}

\newcommand{\rdH}{\mathcal H}

\newcommand{\rdL}{\mathcal{H}^\cbb}

\newcommand{\ccal}{\mathcal{C}}

\newcommand{\ph}{pluri\-harmonic}

\newcommand{\psh}{pluri\-sub\-harmonic}

\newcommand{\qpsh}{$q$-pluri\-sub\-harmonic}

\newcommand{\qpshy}{$q$-pluri\-sub\-harmonicity}

\newcommand{\rqcy}{real $q$-convexity}
\newcommand{\rqc}{real $q$-convex}

\newcommand{\psc}{pseudo\-convex}

\newcommand{\qpsc}{$q$-pseudo\-convex}

\newcommand{\pscy}{pseudo\-con\-vexity}

\newcommand{\CVX}{\mathcal{CVX}}

\newcommand{\C}{\mathcal{C}}

\newcommand{\nbh}{neighborhood}

\newcommand{\cont}{continuous}
\newcommand{\usc}{upper semi-continuous}
\newcommand{\lsc}{lower semi-continuous}
\newcommand{\fct}{function}
\newcommand{\fcts}{functions}

\renewcommand{\and}{\ \hbox{and}\ }
\newcommand{\qand}{\quad \hbox{and} \quad}

\newcommand{\cld}{\overline{D}}
\newcommand{\ol}[1]{\overline{#1}}

\newcommand{\relc}{\Subset}

\newcommand{\repa}{\operatorname{Re}}
\newcommand{\impa}{\operatorname{Im}}

\newcommand{\mfolgt}{$\Rightarrow$}

\newcommand{\D}{\displaystyle}

\newcommand{\eps}{\varepsilon}
\newcommand{\vphi}{\varphi}

\newcommand{\cl}[1]{\ol{#1}}

\newcommand{\fe}{ \ \hbox{for\ every}\ }

\title{On rigid $q$-plurisubharmonic functions and $q$-pseudoconvex tube domains in $\mathbb{C}^n$}

\date{}

\author[1]{Thomas Pawlaschyk}
\affil[1]{Department of Mathematics, University of Wuppertal, 42119 Wuppertal, Germany, \href{mailto:pawlaschyk@uni-wuppertal.de}{pawlaschyk@uni-wuppertal.de}, \href{https://orcid.org/0009-0004-0494-3273}{ORCID: 0009-0004-0494-3273}}

\begin{document}

\maketitle

\begin{abstract} In the spirit of Lelong and Bochner, we show that an upper semi-continuous function defined on a open tube set $\Omega=\omega + i\mathbb{R}^n$ in $\mathbb{C}^n$, where $\omega$ is an open set in $\mathbb{R}^n$, and which is invariant in its imaginary part, is $q$-plurisubharmonic on $\Omega$ (in the sense of Hunt and Murray) if and only if it is real $q$-convex on $\omega$, i.e., it admits the local maximum property with respect to affine linear functions on real $(q+1)$-dimensional affine subspaces. From this, we conclude that, for $a>0$, the set $\omega+i(-a,a)^n$ is $q$-pseudoconvex in $\mathbb{C}^n$ if and only if $\omega$ is a real $q$-convex set in~$\mathbb{R}^n$, i.e., $\omega$ admits a real $q$-convex exhaustion function on $\omega$. We apply these results to complements of graphs of affine linear maps and to Reinhardt domains.
\end{abstract}



\subsection*{Acknowledgement} 

The main part of this paper was essentially developed during the author's research stay at the Postech University of Science and Technology, Pohang, South Korea, supported by the 2016 NRF-DAAD Scientist Exchange Program. Utmost thanks to Kang-Tae Kim for his hospitality during this stay. After a break, the author finalized his studies on the content of this paper at the end of 2024 at the RIMS, Kyoto, and the Osaka Metropolitan University, Osaka, Japan. Thanks a lot to Takayuki Koike for his invitation, hospitality and the discussion which led to Corollary~\ref{Koike-1} and Corollary~\ref{Koike-2}.\footnote{MSC: Primary 32F10; Secondary 26B25}

This work was supported by the Research Institute for Mathematical Sciences (RIMS), an International Usage/Research Center located in Kyoto University.



\section{Introduction}

The classes of convex and plurisubharmonic functions are among the most important families of functions in real and complex analysis in several variables, respectively. Both are closely related, as was thoroughly demonstrated in the classical paper by Lelong~\cite{Le}. On the one hand, every locally convex \fct\ is \psh, but the converse is false in general. On the other hand, an \usc\ \fct\ defined on a tube domain $\Omega=\omega+i\rbb^n$, which is invariant in its imaginary parts, is \psh\ on $\Omega$ if and only if it is locally convex on the open set $\omega$ in $\rbb^n$. From this, Lelong deduced that $\Omega$ is a domain of holomorphy (or, equivalently, \psc) if and only if $\omega$ is a convex set in $\rbb^n$. Lelong further extended this result by replacing the tube domain $\Omega$ with a cylinder of the form $\omega+i(-a,a)^n$ for $a>0$. Earlier results in this direction were obtained by Bochner~\cite{Bo}.

In this paper, we extend Lelong's results to the class of \qpsh\ \fcts\ in the sense of Hunt-Murray~\cite{HM} and \qpsc\ domains in the sense of S\l{}odkowski~\cite{Sl}. For this, we introduce the notion of \rqc\ \fcts\ on open sets in $\rbb^n$. These are \usc\ functions that satisfy the local maximum property with respect to affine linear functions on real $(q+1)$-dimensional affine subspaces. In this sense, they generalize locally convex \fcts\ and serve as the real analogues to weakly $q$-convex functions in the sense of Grauert, in the following way: a $\ccal^2$-\fct\ is \rqc\ if and only if its real Hessian has at most $q$ negative eigenvalues at each point. Moreover, they possess approximation properties similar to those developed by S\l{}odkowski~\cite{Sl2} for \qpsh\ \fcts. Using \rqc\ \fcts, we introduce \rqc\ sets and establish equivalent characterizations of such sets in Theorem~\ref{equivrqc}. From this, we obtain the main results of our paper (Theorem~\ref{rqc-qpsh} and Theorem~\ref{rqc-qpsc}):

\medskip

\noindent\textbf{First Main Theorem.} \textit{Let $\omega$ be an open subset in $\rbb^n$. An \usc\ \fct\ $\psi$  defined on the open tube set $\Omega=\omega+i\rbb^n$ in $\cbb^n$ with $\psi(z)=\psi(\repa{(z)})$ is \qpsh\ if and only if it is \rqc\ on $\omega$.}


\noindent\textbf{Second Main Theorem.} \textit{An open set $\omega$ in $\rbb^n$ is \rqc\ if and only the set $\omega+i(-a,a)^n$ is \qpsc\ in $\C^n$ for some/any $a \in (0,+\infty]$.}

\medskip

The main theorems were already presented in the author's Ph.D. thesis in 2015~\cite{TPTHESIS}, but they were not published in an suitable journal. Recently, in 2024, A. Sadullaev\footnote{The author deeply regrets the unexpected passing of Azimbay Sadullaev (1947-2025), who was a frequent visitor to the complex analysis group in Wuppertal, where he gave several lectures and talks on pluripotential theory.} presented similar results for a different class of generalized convex  functions \cite{Sa} at the conference GMOCA in Wuppertal, Germany. This motivated the author to believe that the results of the present paper might be of interest to experts in several complex variables as well in convexity theory. 

Nevertheless, the discussion on the equivalent notions for \rqc\ sets in Section~\ref{sect-rqc-sets} up to Theorem~\ref{equivrqc}, together with its application to complements of graphs of affine linear maps (Thoerem~\ref{rqc-lin-fol-0}) and to Reinhardt domains (Corollaries~\ref{Koike-1} and~\ref{Koike-3}), is entirely new and have not been published previously.

\section{Real $q$-convex \fcts}\label{sec-rqc}

Throughout this paper, the set $\omega$ denotes an open set in $\rbb^n$. The Euclidean scalar product on $\rbb^n$ is given by $\langle x,y \rangle_2:=\sum_{j=1}^n x_jy_j$ which induces the norm $\|x\|_2=\sqrt{\langle x,x \rangle_2}$ on $\rbb^n$. The boundary distance of a point $p$ in $\omega$ to the boundary $\partial\omega$ of $\omega$ is defined by $d_2(p,\partial\omega)=\inf\{\|x-p\|_2 : x \in \partial\omega\}$. The balls $B^n_r(p)=B_r(p)$ in $\rbb^n$ are given by $B_r(p):=\{x \in \rbb^n : \|x-p\|_2^2<r\}$. 

Especially in this section, we omit most proofs, since they either follow  easily from the definitions or can be found in detial in~\cite{TPTHESIS} for the interested reader.

We begin with the definition of real $q$-convex \fcts\ in the Euclidean space $\rbb^n$, which is based on classical convexity.

\begin{defn} \label{def-rqc} \label{def-subcve} Let $\omega$ be an open set in $\rbb^n$ and let $q\in \{ 0,\ldots,n-1\}$.
\begin{enumerate}



\item  We call an \usc\ \fct\ $u:\omega\to[-\infty,+\infty)$ to be \emph{\rqc}, if, for short, it \emph{fulfills the local maximum property on $\omega$ with respect to affine linear functions on $(q+1)$-dimensional subspaces}, 
i.e., if for every $(q+1)$-dimensional affine subspace $\pi$, every ball $B \relc \omega$ and every affine linear \fct\ $l$ on $\pi$ with $u \leq l$ on $\partial B \cap \pi$ we already have that $u \leq l$ on $\cl{B}\cap \pi$.

\item If $m \geq n$, each \usc\ \fct\ is automatically real $m$-convex by convention.

\end{enumerate}
\end{defn}

The subsequent properties follow immediately from the definition of \rqcy.

\begin{prop} \label{prop-rqc} Let all \fcts\ mentioned below be defined on an open set $\omega$ in $\rbb^n$ with image in $[-\infty,+\infty)$.

\begin{enumerate}

\item If $u$ is real-valued, then it is locally convex if and only if it is real $0$-convex. 


\item Every real $q$-convex \fct\ is real $(q+1)$-convex.

\item \label{prop-rqc-cone} If $\lambda \geq0 $, $c \in \rbb$, and $u$ is \rqc, then $\lambda u + c$ is also \rqc.

\item \label{prop-rqc-decr} The limit of a decreasing sequence $\{u_k\}_{k \in \nbb}$ of \rqc\ \fcts\ is again \rqc.

\item \label{prop-rqc-sup} If $\{ u_i \}_{i \in I}$ is a family of locally bounded \rqc\ \fcts, then the upper semi-continuous regularization $u^*(x):=\limsup_{y\to x} u(y)$ of $u:=\sup_{i \in I} u_i$ is \rqc. In particular, the maximum of finitely many \rqc\ \fcts\ is again \rqc.

\item \label{prop-rqc-lin-change} A \rqc\ \fct\ remains \rqc\ after a linear change of coordinates.

\item \label{prop-rqc-lin} An \usc\ \fct\ $u$ is \rqc\ if and only if $u+l$ is \rqc\ for every affine linear \fct\ $l$ on $\rbb^n$.

\end{enumerate}

\end{prop}

The next statement corresponds essentially to Lemma~4.5 in~\cite{Sl2}.

\blem \label{lem-strike} Let $X$ be a vector space over the field $\kbb \in \{\rbb,\cbb\}$ equipped with the inner product $\langle \cdot,\cdot\rangle$. Let $\|\cdot\|$ denote its induced norm and let $u$ be an \usc\ \fct\ on a compact set $K$ in $X$. Suppose that there is another compact set $L$ in $K$ with $\max_L u < \max_K u $. Then there are a point $p$ in $K\setminus L $, a real number $\eps>0$ and an $\rbb$-linear \fct\ $l:X\to\rbb$ such that
\[u (p)+l(p)=0 \qand u(x)+l(x) < -\eps\|x-p\|^2 \fe x \in K\setminus \{ p\}.\]
\elem

From the preceding lemma, we conclude that real $q$-convexity is a local property.

\begin{cor} \label{loc-rqc} Let $u$ be \usc\ on an open set $\omega$ in $\rbb^n$. Then $u$ is \rqc\ on $\omega$ if and only if it is locally \rqc\ on $\omega$, i.e., for every point $p$ in $\omega$ there is a \nbh\ $V$ of $p$ in $\omega$ such that $u$ is \rqc\ on~$V$.
\end{cor}


Lemma~\ref{lem-strike} has another important consequence for \rqc\ \fcts.

\begin{thm}[Maximum principle] \label{rqc-max-prop} Let $q \in \{ 0,\ldots,n-1\}$ and let $\omega$ be a relatively compact open set in $\rbb^n$. If $u$ is \rqc\ on $\omega$ and \usc\ up to the closure of $\omega$, then \[
\max\{u(x): x \in \cl{\omega}\}=\max\{u(x): x \in \partial\omega\}.
\]
\end{thm}


Using the maximum principle, two \rqc\ \fcts\ can be patched together to obtain a new \rqc\ \fct.

\begin{thm} \label{glue-rqc} Let $\omega_1$ and $\omega$ be two open sets in $\rbb^n$ with $\omega_1 \subset \omega$. Let $u$ be a \rqc\ \fct\ on $\omega$ and $u_1$ be a \rqc\ \fct\ on $\omega_1$ such that
\bean \label{glue-rqc-eq01}
& \D \limsup_{\substack{y\to x\\ y \in \omega_1}} u_1(y) \leq u(x) \fe x \in \partial\omega_1 \cap \omega. &
\eean
Then the following \fct\ is \rqc\ on $\omega$,
\[\psi(x):=\left\{\ba{ll} \max\{ u(x),u_1(x)\}, & x \in \omega_1\\ u(x), & x \in \omega\setminus \omega_1 \ea\right\} .\]
\end{thm}

\begin{proof} It is obvious that the \fct\ $\psi$ is \usc\ on $\omega$. Let $\pi$ be a real $(q+1)$-dimensional affine subspace in $\rbb^n$, $B$ be a ball lying relatively compact in $\pi \cap \omega$ and let $l$ be an affine linear \fct\ on $\pi$ such that $\psi \leq l$ on $\partial B$. Since $\psi$ coincides with $u$ on $\omega\setminus\cl{\omega_1}$ and since it is a maximum of the two \rqc\ \fcts\ $u$ and $u_1$ on $\omega_1$, $\psi$ is \rqc\ on $\omega\setminus \partial\omega_1$. Thus, we can assume that $B \cap \partial\omega_1 \neq \emptyset$. Since $u$ is \rqc\ on $\omega$ and by the inequalities $u \leq \psi \leq l$ on $\partial B$, we obtain that $u \leq l$ on $B$. Therefore, we have that $\psi =u \leq l$ on $B\cap(\omega\setminus\omega_1)$. In particular, we have that $\psi=u \leq l$ on $B \cap \partial\omega_1$. This implies that $\psi \leq l$ on $\partial(B\cap \omega_1)$. Since $\psi$ is \rqc\ on $\omega_1$, the maximum principle from the previous theorem yields $\psi \leq l$ on $B\cap \omega_1$. By the previous discussion, we have that $\psi \leq l$ on $B$. Finally, we can conclude that $\psi$ is \rqc\ on~$\omega$.
\end{proof}

Next, we provide another characterization of real $q$-convexity in terms of eigenvalues of its real Hessian. Before that, we define real $q$-convex functions that are stable under small perturbations by convex functions.

\begin{defn} Let $\omega$ be an open set in $\rbb^n$. We say that an \usc\ \fct\ $u$ on $\omega$ is \emph{strictly \rqc} if for every point $p$ in $\omega$ there exist a \nbh~$U$ of $p$ and a positive number $\eps_0>0$ such that $x\mapsto u(x)+\eps\|x-p\|_2^2$ is \rqc\ on~$U$ for every $\eps \in (-\eps_0,\eps_0)$.
\end{defn}


In the case of $\ccal^2$-smooth \fcts, we have the following characterization of (strict) \rqcy.

\begin{thm}\label{smooth-rqc} Let $q \in \{0,\ldots,n-1\}$ and $\omega$ be an open set in $\rbb^n$. A $\ccal^2$-smooth \fct\ $u$ on $\omega$ is (strictly) real $q$-convex if and only if for every point $p \in \omega$ the real Hessian $\rdH_u(p)=\left( \frac{\partial^2 u}{\partial x_k \partial x_l}(p)\right)_{k,l=1}^n$ of $u$ at $p$ has at most $q$ negative (non-positive, resp.) eigenvalues.
\end{thm}

\begin{proof} By Corollary~\ref{loc-rqc}, \rqcy\ is a local property, so all considerations can be made in a small \nbh\ of some fixed point $p \in \omega$. Due to Proposition~\ref{prop-rqc}~(\ref{prop-rqc-cone}) and~(\ref{prop-rqc-lin-change}), we can assume without loss of generality that $p=0$, $u(p)=0$ and that $u$ has the following Taylor expansion in some \nbh\ of the origin,
\[u(x) = A(x) + \frac{1}{2} x^t\rdH_u(0)x + o(\|x\|^2_2),\]
where $A(x)=\nabla u(0)x$ is considered as a linear \fct\ $\rbb^n \to \rbb$. According to Proposition~\ref{prop-rqc}~(\ref{prop-rqc-lin}), by replacing $u$ by $u-A$, we can further assume without loss of generality that $u$ has the following form near the origin,
\[u(x)=\frac{1}{2} x^t\rdH_u(0)x + o(\|x\|^2_2).\]
Now if the real Hessian of $u$ has at least $q+1$ negative eigenvalues at the origin, then we can find a real $(q+1)$-dimensional affine subspace $\pi$ in $\rbb^n$ and a ball $B$ inside $\pi \cap \omega$ such that $u$ is strictly negative at every point on the boundary of $B$ but vanishes inside $B$ at the origin. Thus, in view of the maximum principle, it cannot be \rqc\ on~$\omega$.

On the other hand, if $u$ is not \rqc, then there are a point $p_0 \in \omega$, a real $(q+1)$-dimensional affine subspace $\pi$, a ball $B$ in $\pi \cap \omega$ containing $p_0$ and an affine linear \fct\ $l_1$ on $\pi$ such that $u(x) \leq l_1(x)$ for every $x \in \partial B$, but $u(p_0)>l_1(p_0)$. Then by Lemma~\ref{lem-strike} there are a point $p_1$ inside $B$, a positive number $\eps>0$ and another  linear \fct\ $l_2$ on $\pi$ such that
\[u(p_1)-l_1(p_1)-l_2(p_1)=0 \qand u(x)-l_1(x)-l_2(x) < -\eps\|x-p_1\|^2_2.\]
for every $x\in \cl{B}\setminus\{ p_1\}$. Hence, the \fct\ $u-l_1-l_2$ attains a strict local maximum at~$p_1$. Therefore, the real Hessian of $u$ at $p_1$, which corresponds to the real Hessian of $u-l_1-l_2$ at $p_1$, has at least $q+1$ negative eigenvalues.
\end{proof}

Theorem~\ref{smooth-rqc} allows us easily to construct examples of real $q$-convex \fcts\ in $\rbb^n$.

\begin{ex}\label{ex-rqc} Consider the subsequent functions defined on $\rbb^2$.

\begin{enumerate} 

\item The functions $u(x,y)=-x^2$ and $v(x,y)=-y^2$ are both real 1-convex, but their sum $(u+v)(x,y)=-x^2-y^2$ is \textbf{not} 1-convex.

\item\label{ex-rqc-2} The real 1-convex \fcts\ $v_{n}(x,y)=-nx^2$ decrease point-wise for $n \to \infty$ to $v(x)=\left\{ \begin{array}{ll} 0, & x =0 \\ -\infty, & x \neq 0 \end{array}\right\}$, which is real 1-convex due to Proposition~\ref{prop-rqc}~(\ref{prop-rqc-decr}). 

\item By the same argument, the characteristic \fct\ $\Chi_S=\left\{ \begin{array}{ll} 1, & x \in S \\ 0, & x \notin S \end{array}\right\}$ of the real line $S=\{ (x,0) \in \rbb^2 : x \in \rbb\}$ in $\rbb^2$ is real 1-convex (as a decreasing limit of the real 1-convex \fcts\ $w_n(x,y)=e^{-nx^2}$.

\item This demonstrates that, in general, real-valued real $q$-convex \fcts\ are not necessarily \cont, if $q \geq 1$, whereas every real-valued 0-convex, i.e., locally convex \fct, is \cont\ (see Theorem 10.1 in \cite{Rocka}).

\end{enumerate}

\end{ex}

Motivated by the previous examples, we can construct further \rqc\ \fcts.

\begin{lem}\label{rqc-on-fol} Let $q \in \{0,1,\ldots,n-1\}$ and let $\{\pi_\alpha\}_{\alpha \in A}$ be a collection of real $(n-q)$-dimensional affine subspaces $\pi_\alpha$ in $\rbb^n$ such that $\bigcup_{\alpha \in A} \pi_\alpha=\rbb^n$. Let $u$ be a \cont\ \fct\ on an open set $\omega$ in $\rbb^n$ such that $u$ is locally convex on each intersection $\pi_\alpha \cap \omega$, $\alpha \in A$. Then $u$ is \rqc\ on $\omega$.
\end{lem}

\begin{proof} Observe that if $\pi=\rbb^{n-q}\times\{0\}^q$, then by similar arguments as in Example~\ref{ex-rqc}~(\ref{ex-rqc-2}), we can show that 
\[
v_\pi(x)=\left\{ \begin{array}{ll} 0, & x \in \pi \\ -\infty, & x \notin \pi \end{array}\right\}
\]
is \rqc\ on $\rbb^n$. Since \rqcy\ is invariant under linear changes of coordinates, we have that $v_{\pi_{\alpha}}$ is \rqc\ on $\rbb^n$ for each $\alpha \in A$.

Now if $u$ is locally convex on $\pi_{\alpha} \cap \omega$, we can extend $u$ to a locally convex function $\hat{u}_\alpha$ defined on open \nbh\ $U$ of $\pi_{\alpha} \cap \omega$ in $\omega$. By Proposition~\ref{prop-rqc-lin}~(7), the sum
\[
u_\alpha(x):=(\hat{u}_\alpha+v_{\pi_\alpha})(x)=\left\{ \begin{array}{ll} u, & x \in \pi_\alpha \\ -\infty, & x \notin \pi_\alpha \end{array}\right\}
\]
is \rqc\ on $\omega$. Finally, observe that $u=\sup_{\alpha \in A}u_\alpha$, so that $u$ is \rqc\ on $\omega$ due to Proposition~\ref{prop-rqc}~(\ref{prop-rqc-sup}).
\end{proof}

Theorem \ref{smooth-rqc} also yields a technique similar to Lemma \ref{lem-strike}, which we will use later.

\blem\label{lem-smooth-rqc} Let $\omega$ be an open set in $\rbb^n$. Assume that $u$ is not \rqc\ on~$\omega$. Then there is a ball $B \relc \omega$, a point $x_1 \in B$, a number $\eps>0$ and a $\ccal^\infty$-smooth real $(n{-}q{-}1)$-convex \fct\ $v$ on $\rbb^n$ such that
\[(u+v)(x_1) = 0 \qand (u+v)(x) < -\eps\|x-x_1\|^2_2 \quad \fe x \in B\setminus\{x_1\}.\]
\elem

\begin{proof} Since $u$ is not \rqc\ on $\omega$, there exist a ball $B\relc\omega$, a point $x_0$ in $B$, a $(q+1)$-dimension affine subspace $\pi$ and an affine linear \fct\ $l:\rbb^n\to\rbb$ such that $u+l< 0$ on $\partial B\cap\pi$ and $u(x_0)+l(x_0)>0$. Let $h:\rbb^n\to\rbb^{n-q-1}$ be a linear map such that $\pi=\{h=0\}$ and fix a number $c>0$. In view of Theorem~\ref{smooth-rqc}, it is easy to verify that the $\ccal^{\infty}$-smooth \fct\ $v_c(x):=l(x)-c\|h(x)\|^2_2$ is real $(n{-q-}1)$-convex on $\rbb^n$. Moreover, it equals $l$ on $\pi$ and tends to $-\infty$ outside $\pi$ when $c$ goes to $+\infty$. Therefore, if we choose $c$ large enough, then we can arrange that $u+v_c < 0$ on $\partial B$ and $u(x_0)+v_c(x_0)>0$. Now it follows from Lemma~\ref{lem-strike} that there is another linear \fct\ $l_1:\rbb^n\to\rbb$, a point $x_1 \in B$ and $\eps>0$ such that $(u+v_c+l_1)(x_1)=0$, but $(u+v_c+l_1)(x)<-\eps\|x-x_1\|^2_2$ for every $x \in B\setminus\{x_1\}$. Finally, $v:=v_c+l_1$ is the demanded \fct\ in view of Proposition~\ref{prop-rqc}~(\ref{prop-rqc-lin}).
\end{proof}

\section{Approximation of real $q$-convex functions}\label{sec-reg-rqc}

We present an approximation method for real $q$-convex functions by almost everywhere twice differentiable ones. It is based on the ideas developed by S\l{}odkowski's in~\cite{Sl2}.

\begin{thm}[Busemann-Feller-Alexandroff, cf.~\cite{BCP}]\label{b-f-a} \label{thm-bfa} Let $u$ be a real-valued locally convex \fct\ on an open set $\omega$ in $\rbb^n$. Then, almost everywhere on $\omega$, the \fct\ $u$ is twice differentiable and its gradient $\nabla u$ is differentiable.
\end{thm}

This important theorem motivates the introduction of the following family of \fcts.

\begin{defn}\label{rem-peano-diff} Let $\omega$ be an open set in $\rbb^n$ and $L \geq 0$.

\begin{enumerate}

\item The symbol $\ccal_L^1(\omega)$ is the family of all real valued \fcts\ $g$ on $\omega$ such that $u(x):=g(x)+\frac{1}{2}L\|x\|^2_2$ is locally convex on $\omega$.

\item Let $g$ be a \fct\ in $\ccal_L^1(\omega)$. In view of the Busemann-Feller-Alexandroff theorem, the real Hessian $\rdH_g(x)$ of $g$ exists at almost every point $x$ in $\omega$. At these points, the smallest eigenvalue is bounded from below by $-L$. It is therefore reasonable to say that \fcts\ in $\ccal^1_L(\omega)$ have a \emph{lower bounded Hessian}.

\item The collection of all \fcts\ on $\omega$ with lower bounded Hessian is denoted by $\ccal^1_\bullet(\omega)$.

\end{enumerate}
\end{defn}

Integral convolution provides an important method to approximate convex \fcts, but it is not suitable for real $q$-convex functions. An alternative is given by a convolution method based on taking a supremum rather than an integral.

\begin{defn} Let $u,v$ be two non-negative \fcts\ defined on possibly different subsets of $\rbb^n$. Then for every $x \in \rbb^n$ the \emph{supremum convolution} of $u$ and $v$ is defined by
\[(u \ast v)(x):=\sup\{ \hat u(y)\hat v(x-y):y \in \rbb^n\},\]
where $\hat u$ and $\hat v$ denote the trivial extensions of $u$ and $v$ by zero into the whole space~$\rbb^n$.
\end{defn}

Applying the supremum convolution to \fcts\ with lower bounded Hessian, we obtain the following statement (cf.~Proposition 2.6 in \cite{Sl2}).

\begin{prop}\label{conv-usc-peano} Let $M>0$ be a positive number. Let $u$ and $g$ be two non-negative bounded \usc\ \fcts\ on $\rbb^n$. If $g \in \ccal_L^1(\rbb^n)$, then $u \ast g$ lies in $\ccal^1_{ML}(\rbb^n)$, where $M:=\sup\{u(x): x \in \rbb^n\}$. In particular, $u \ast g$ is \cont\ on~$\rbb^n$ and twice differentiable almost everywhere on~$\rbb^n$.
\end{prop}

Our next goal is to characterize twice differentiable \rqc\ \fcts\ by a certain quantity that represents exactly the largest eigenvalue of the real Hessian of a $\ccal^2$-smooth \fct\ at a given point.



\begin{thm}\label{thm-max-ev} If $u$ is a locally convex \fct\ on an open set $\omega$ in $\rbb^n$ such that for the maximal eigenvalue of the Hessian of $u$ at $x$, 
\[\lambda_u(x):= 2\limsup_{\eps \to 0} (\max\{u(x+\eps h)-u(x)-\eps\nabla u(x)h : h \in \rbb^n, \ \|h\|_2=1\})/\eps^2,\]
we have that $\lambda_u(x) \geq M$ for almost every $x \in \omega$, then $\lambda_u(x)\geq M$ for every $x \in \omega$.
\end{thm}

The preceding statements permit us to generalize Theorem~\ref{smooth-rqc} to twice differentiable \rqc\ \fcts.

\begin{thm}\label{prop-rqc-twice} Let $q\in\{ 0,\ldots,n-1\}$ and let $u$ be \usc\ on an open set $\omega$ in $\rbb^n$.

\begin{enumerate}

\item If $u$ is \rqc\ on $\omega$ and twice differentiable at a point $p$ in $\omega$, then the real Hessian of $u$ at $p$ has at most $q$ negative eigenvalues.

\item If $u \in \ccal^1_L(\omega)$ and its real Hessian at almost every point in $\omega$ has at most~$q$ negative eigenvalues, then $u$ is \rqc\ on $\omega$.

\end{enumerate}
\end{thm}

\begin{proof} 1.~Pick a point $p$ in $\omega$ such that $\rdH_u(p)$ exists. Let $B_r(p) \relc \omega$ be a ball centered in $p$ with radius $r>0$. Then for $t\in(0,1)$ the \fct\ $u_t$ given by
\[B_r(0) \ni x \ \mapsto \ u_t(x):=\big( u(p+tx)-u(p)-t\langle\nabla u(p),x\rangle\big)/t^2\]
is \rqc\ on $B_r(0)$ due to Proposition~\ref{prop-rqc}~(\ref{prop-rqc-cone}) and~(\ref{prop-rqc-lin}). Since $u$ is twice differentiable at $p$, the family $\{u_t\}_{t\in(0,1)}$ tends uniformly to $x \mapsto u_0(x):=x^t \rdH_u(p)x$ in a small \nbh\ of the origin as $t$ tends to zero. Therefore, the \fct\ $u_0$ is \rqc\ and $\ccal^2$-smooth on a \nbh\ of the origin. By Theorem~\ref{smooth-rqc} the real Hessian of $u_0$ at the origin has at most $q$ negative eigenvalues. Since $\rdH_{u_0}(0)=\rdH_u(p)$, the proof of the first statement is finished.

\medskip

2.~If $u$ is not \rqc\ on $\omega$, then it follows from Lemma~\ref{lem-smooth-rqc} that, without loss of generality, there exist a ball $B_r(0)\relc\omega$, a number $\eps>0$ and a $\ccal^\infty$-smooth real $(n{-}q{-}1)$-convex \fct\ $v$ on $\rbb^n$ which satisfies $(u+v)(0)=0$ and
\bea \label{thm-rqc-twice-000}
(u+v)(x) < -\eps\|x\|^2_2 \fe x \in \cl{B_r(0)}\setminus\{0\}.
\eea
Recall that $u \in \ccal^1_L(\omega)$ and define
\[f:=u+v, \quad M_v:=\sup\{\lambda_v(x):x \in \cl{B_r(0)}\} \quad \and \quad M:=L+M_v.\]
Then $f$ is non-positive and belongs to $\ccal^1_{M}(\omega)$, so $g(x):=f(x)+\frac{1}{2}M\|x\|^2_2$ is convex on $B_r(0)$. Therefore, for every $x \in \cl{B_r(0)}$ we have that
\[0=2g(0) \leq g(x)+g(-x)=f(x)+f(-x) + M\|x\|^2_2 \leq f(x) + M\|x\|^2_2.\]
Thus, $-M\|x\|^2_2 \leq f(x)$. On the other hand, $f(x) \leq -\eps\|x\|^2_2$, so the gradient of $f$ at $0$ exists and vanishes there. Of course, the same is also true for the \fct\ $g$. Thus, in view of property~\eqref{thm-rqc-twice-000}, we can estimate the maximal eigenvalue of $g$ at 0 as follows:
\bea \label{thm-rqc-twice-001}
&\D \lambda_g(0)= 2\limsup_{\eps \to 0} \big(\max\{g(\eps h) : h \in \rbb^n, \ \|h\|_2=1\}\big)/\eps^2 \leq M-2\eps.&
\eea
By the Busemann-Feller-Alexandroff theorem (see Theorem~\ref{thm-bfa}), the real Hessian of~$f$ exists almost everywhere on $\omega$. Moreover, since $\rdH_u$ has at most $q$ negative and $\rdH_v$ has at most $n{-q-}1$ negative eigenvalues, the real Hessian of the sum $f=u+v$ has at least one non-negative eigenvalue almost everywhere on $\omega$. Therefore, since the the largest eigenvalue of the \fct\ $x\mapsto\frac{1}{2}M\|x\|^2_2$ is exactly $M$, we derive the estimate $\lambda_g(x) \geq M$ at almost every point in $B_r(0)$. Then it follows from Theorem~\ref{thm-max-ev} that $\lambda_g \geq M$ everywhere on $B_r(0)$. In particular, $\lambda_g(0) \geq M$, which is a contradiction to~\eqref{thm-rqc-twice-001}.
\end{proof}

We show that any \rqc\ \fct\ can be approximated from above by a decreasing sequence of \rqc\ \fcts\ being \cont\ everywhere and twice differentiable almost everywhere. 

\begin{thm}\label{conv-rqc-peano} Let $u$ be a non-negative bounded \rqc\ \fct\ on an open set $\omega$ in $\rbb^n$. Let $g\in\ccal^1_L(\rbb^n)$ be a non-negative \fct\ with compact support in some ball $B_r(0)$. Define the set $\omega_r:=\{ x \in \omega : d_2(x,\partial\omega)>r\}$ and the number $M_r:=\sup\{u(x) : x \in \omega_r\}$. Then $u \ast g$ lies in $\ccal_{LM_r}^1(\rbb^n)$ and it is \rqc\ on~$\omega_r$.
\end{thm}

\begin{proof} Recall that $\hat{u}$ denotes the trivial extension of $u$ by zero to the whole of $\rbb^n$. The supremum convolution of $u$ and $g$ at $x \in \omega_r$ can be rewritten as follows,
\bean
(u \ast g)(x)& = &\sup\{ \hat u(y)g(x-y):y \in \rbb^n\}\\
&=&\sup\{ \hat u(x-t)g(t):t \in \rbb^n\}\\
&=&\sup\{ u(x-t)g(t):t \in B_r(0)\}.
\eean
It follows from Proposition~\ref{prop-rqc}~(\ref{prop-rqc-cone}) and~(\ref{prop-rqc-lin-change}) that $x \mapsto g(t)u(x-t)$ is \rqc\ on $\omega_r$ for every $t \in B_r(0)$. Since, in view of Remark~\ref{rem-peano-diff} and Proposition~\ref{conv-usc-peano}, the \fct\ $u \ast g$ is \cont, Proposition~\ref{prop-rqc}~(\ref{prop-rqc-sup}) implies that $u \ast g$ is \rqc\ on $\omega_r$. Finally, it follows directly from Proposition~\ref{conv-usc-peano} that $u \ast g$ belongs to $\ccal_{LM_r}^1(\rbb^n)$.
\end{proof}

This leads to the following important approximation technique.

\begin{prop}\label{rqc-sl-approx-1} Let $u$ be a \rqc\ \fct\ on an open set $\omega$ in $\rbb^n$ and let $D$ be a relatively compact open set in $\omega$. Assume that $f$ is a \cont\ \fct\ on $\omega$ and satisfies $u<f$ on a \nbh\ of $\cld$. Then there is a positive number $L>0$ and a \cont\ \fct\ $\tilde u \in \ccal^1_{L}(\rbb^n)$ which is \rqc\ in a \nbh\ of $\cld$ and which fulfills $u < \tilde u < f$ on $\cld$.
\end{prop}

\begin{proof} Let $r$ be a positive number so small that that $\cld$ is contained in $D_r:=\omega_r\cap B_{1/r}(0)$, where $\omega_r:=\{ x \in \omega : d_2(x,\partial\omega)>r\}$. Given $k \in \nbb$, we set $v:=\max\{u,-k\}+k+1/k$. Then $u<v-k$ and $v$ is positive. Since the sequence $(v-k)_{k \in \nbb}$ decreases to $u$, we can find a large enough integer $k \in \nbb$ such that $v-k < f$ on $\cl{D}$. By upper semi-continuity of $v$ and compactness of $\cld$, we can choose another radius $r'\in(0,r)$ so small that $D \relc \omega_{r'}$ and
\[
\sup\{ v(y)-k:y\in B_{r'}(x)\} < f(x) \quad \fe x \in \cld.
\]
Now pick a $\ccal^\infty$-smooth \fct\ $g$ with compact support in the ball $B_{r'}(0)$ such that $0\leq g \leq 1$ and $g(0)=1$. We set $\tilde u(x):=(v \ast g)(x)-k$ for $x \in \omega$. Then we obtain for every $x \in \cld$ that
\bean
u(x) &< & v(x)-k\\
& =& v(x)g(0)-k\\
& \leq &  \sup\{ v(y)g(x-y):y\in B_{r'}(x)\}-k\\
&=& (v \ast g)(x)-k\\
&=& \tilde u (x) \\
& \leq & \sup\{ v(y):y\in B_{r'}(x)\}-k\\
&=& \sup\{ v(y)-k:y\in B_{r'}(x)\} < f(x).\\
\eean

The rest of the properties of $\tilde u$ follow now from the previous Theorem~\ref{conv-usc-peano}.
\end{proof}

As a consequence, we obtain an approximation property for \rqc\ \fcts\ by twice differentiable ones.

\bcor\label{rqc-sl-approx} Let $\omega$ be an open set in $\rbb^n$, let $K$ be a compact set in $\omega$ and let $u$ be a real $q$-convex function on $\omega$. Then there exists a sequence $\{u_k\}_{k \geq1}$ of \fcts\ $u_k \cap \ccal^1_\bullet(\rbb^n)$ which are real $q$-convex functions near $K$ and decrease on $K$ to $u$. In particular, $u_k$ are \cont\ on $K$ and twice differentiable almost everywhere on~$K$.
\ecor



As an application of Theorem~\ref{prop-rqc-twice} and Corollary~\ref{rqc-sl-approx}, we obtain a result concerning sums of \rqc\ \fcts. This result was proved in~\cite{Sl2} for \qpsh\ \fcts.

\begin{thm}\label{sum-rqc} Given a \rqc\ \fct\ $u_1$ and a real $r$-convex \fct\ $u_2$ on an open set $\omega$ in $\rbb^n$, their sum $u_1+u_2$ is real $(q+r)$-convex on $\omega$.
\end{thm}

\begin{proof} By the previous Theorem~\ref{rqc-sl-approx} and since real $q$-convexity is a local property, we can assume that $u_1$ and $u_2$ have lower bounded Hessian and that they are twice differentiable almost everywhere on $\omega$. Then in view of the first statement of Theorem~\ref{prop-rqc-twice}, the real Hessian of $u_1$ has at most $q$ and the real Hessian of $u_2$ has at most $r$ negative eigenvalues at almost every point in $\omega$. Now it is easy to verify that the sum of the Hessians of $u_1$ and $u_2$ have at most $q+r$ negative eigenvalues almost everywhere. Since the sum $u_1+u_2$ certainly also has lower bounded Hessian and is twice differentiable almost everywhere on $\omega$, it follows from the second statement in Theorem~\ref{prop-rqc-twice} that $u_1+u_2$ is real $(q+r)$-convex on $\omega$.
\end{proof}

It is worth mentioning that there also exists an approximation technique based on piecewise smooth functions. Since we will not use it in this paper, we refer to~\cite{TPTHESIS} for a detailed proof and~\cite{Bu} for its original idea. 

\begin{thm}\label{rqc-bu-approx} Let $\omega$ be an open set in $\rbb^n$. Then for every \cont\ \rqc\ \fct\ $u$ there exists a sequence $\{u_k\}_{k \geq1}$ of \rqc\ \fcts\ with corners on $\omega$ which are locally the maximum of $\mathcal{C}^2$-smooth \rqc\ ones decreasing point-wise to~$u$.
\end{thm}

\section{Real $q$-convex and $q$-plurisubharmonic functions}\label{sec-rqc-qpsh}

We give the the definition and basic properties of \qpsh\ \fcts\ in the sense of Hunt-Murray \cite{HM}. It turns out that they are closely related to real $q$-convex functions in the same way as \psh\ \fcts\ are related to convex \fcts~\cite{Le}.

\bdefn \label{def-qpsh}\label{defn-sph} Let $q\in\{ 0,\ldots,n-1\}$ and let $\psi$ be an \usc\ \fct\ on an open set $\Omega$ in $\cbb^n$.
\begin{enumerate}



\item The \fct\ $\psi$ is \emph{\qpsh}\ on $\Omega$ if it \emph{fulfills the local maximum property on $\Omega$ with respect to \ph\ functions on complex $(q+1)$-dimensional subspaces}, i.e., for every complex $(q+1)$-dimensional affine subspace $\Pi$, every ball $B \relc \Omega$ and every \ph\ \fct\ $h$ on defined in the neighborhood of $\cl{B}$ with $\psi \leq h$ on $\partial B \cap \Pi$ we already have that $u \leq l$ on $\cl{B}\cap \Pi$.\footnote{This type of function was called \emph{pseudoconvex of order $n-q$} by O.~Fujita~\cite{Fu2}. Smooth \qpsh\ \fcts\ are exactly the \emph{weakly $(q+1)$-convex} ones in the sense of Grauert.}


\item If $m \geq n$, every \usc\ \fct\ on $\Omega$ is by convention $m$-\psh.



\end{enumerate}
\edefn

The following properties and results are derived from Hunt-Murray's paper~\cite{HM}. For additional properties, we refer to~\cite{Dieu} and~\cite{TPTHESIS}.

\begin{prop} \label{prop-qpsh} \

\begin{enumerate}

\item \label{prop-qpsh-psh-sph} The $0$-\psh\ \fcts\ are exactly the \psh\ \fcts.

\item It follows directly from the definition of \qpshy\ that a \fct\ $\psi$ is \qpsh\ on an open set $\Omega$ in $\cbb^n$ if and only if $\psi+\varphi$ is \qpsh\ for every pluriharmonic function $h$ on $\Omega$.

\item A \fct\ is \qpsh\ if and only if it is locally \qpsh.

\item \label{prop-qpsh-hol-change} A \qpsh\ \fct\ remains \qpsh\ after a holomorphic change of coordinates.

\end{enumerate}

\end{prop}

We have the following characterization of smooth \qpsh\ \fcts.

\bthm\label{smooth-qpsh} Let $q \in \{0,\ldots,n-1\}$ and let $\psi$ be a $\ccal^2$-smooth \fct\ on an open subset $\Omega$ in $\cbb^n$. Then $\psi$ is \qpsh\ if and only if the complex Hessian $\rdL_\psi(p)=\left( \frac{\partial^2 \psi}{\partial z_k \partial \overline{z}_l}(p)\right)_{k,l=1}^n$ has at most~$q$ negative eigenvalues at every point $p$ in~$\Omega$.
\ethm

The maximum principle holds for \qpsh\ \fcts.

\begin{thm}[Maximum principle] \label{prop-qpsh-locmax} Let $q\in\{ 0,\ldots,n-1\}$ and $\Omega$ be a relatively compact open set in $\cbb^n$. Then any \fct\ $u$ which is \usc\ on $\cl{\Omega}$ and \qpsh\ on $\Omega$ fulfills
\[
\max\{ \psi(z) : z \in \cl{\Omega}\} = \max\{\psi(z) : z \in \partial \Omega\}.
\]
\end{thm}





As a first step toward proving our main results, we show that \rqc\ \fcts\ are indeed \qpsh.

\bthm \label{rqc-qpsh-0} Let $\Omega$ be an open subset in $\cbb^n=\rbb^{2n}$. Then every real $q$-convex \fct\ $u$ on $\Omega$ is \qpsh.
\ethm

\bproof If $q \geq n$, then the statement is trivial, since every \usc\ \fct\ on $\Omega$ is \qpsh\ by convention. Otherwise, by Theorem~\ref{rqc-sl-approx}, we can locally approximate $u$ by a sequence of \rqc\ \fcts\ which are twice differentiable almost everywhere. Thus, since \qpshy\ is a local property, we can assume without loss of generality that $u$ is twice differentiable almost everywhere on $\Omega$. Since $u$ is \qpsh\ on $\Omega$ if and only if it is \qpsh\ on every complex affine subspace of dimension $q+1$, and since the restriction of a \rqc\ \fct\ to an affine subspace clearly remains \rqc, it is enough to prove the statement in the case of $q=n-1$.

Thus, let us assume that $q=n-1$ and that the real Hessian $\rdH_u(p)$ of $u$ at $p$ exists for some point $p$ in $\Omega$. By Theorem~\ref{prop-rqc-twice} (1), the real Hessian $\rdH_u(p)$ of $u$ at $p$ has at least $2n-(n-1)=n+1$ non-negative eigenvalues. This means that there is a real $n+1$ dimensional subspace $V$ of $\cbb^n=\rbb^{2n}$ such that $\rdH_u(p)$ is positive semi-definite on $V$. Since $V$ is not totally real, there is a vector $v$ in $V$ such that $iv$ also lies in $V$. Therefore, since $v^t\rdH_u(p)v$ and $(iv)^t\rdH_u(p)(iv)$ are both non-negative by assumption, it follows that the complex Hessian of $u$ at $p$ is non-negative due to the following identity,
\[
\overline{v}^t\rdL_u(p)v = \frac{1}{4}\Big( v\rdH_u(p)v+(iv)^t\rdH_u(p)(iv)\Big).
\]
Hence, the Levi matrix $\rdL_u(p)$ of $u$ at $p$ has at least one non-negative eigenvalue. By the choice of $p$, we deduce that $\rdL_u$ has at least one non-negative eigenvalue almost everywhere on $\Omega$. Then Theorem 4.1 in~\cite{Sl2} implies that the \fct\ $u$ is $(n-1)$-\psh\ on~$\Omega$.
\eproof

The previous result cannot be improved.

\bex Consider the \fct\ $z\mapsto\repa(z)^2 - \impa(z)^2=\repa(z^2)$. It is harmonic on $\cbb$ (i.e., 0-\psh), but not locally convex (i.e, real 0-convex).
\eex

However, under certain additional assumptions, we obtain a converse statement to Theorem~\ref{rqc-qpsh-0}. For this, we have to restrict to \fcts\ that are invariant in their imaginary parts.

\begin{defn} Let $\omega$ be an open set in $\rbb^n$.
\begin{enumerate}

\item A \fct\ $\psi=\psi(z)$ on a \emph{tube set} $\omega+i\rbb^n$ in $\cbb^n$ is called \emph{rigid} if $\psi(z)=\psi(\repa(z))$ for every $z \in \omega+i\rbb^n$.

\item By the definition, a rigid \fct\ $\psi$ on a tube set $\omega+i\rbb^n$ can be naturally considered as a \fct\ $x \mapsto \psi(x)$ on $\omega$. On the other hand, every \fct\ $u$ on $\omega$ induces a well defined rigid \fct\ on $\omega+i\rbb^n$ via $z \mapsto u(\repa(z))$ for every $z \in \omega+i\rbb^n$.

\end{enumerate}
\end{defn}

We generalize Lelong's observation~\cite{Le} that every rigid \psh\ \fct\ is locally convex (case $q=0$) to the general case $q\geq0$.

\bthm[First main theorem] \label{rqc-qpsh} Let $\omega$ be an open set in $\rbb^n$. Then every rigid \fct\ on $\Omega=\omega+i\rbb^n$ is \qpsh\ if and only if it is \rqc\ on $\omega$.
\ethm

\bproof Using the approximation techniques for \rqc\ \fcts\ from Section~\ref{sec-reg-rqc} and by counting the eigenvalues of the involved Hessians, we can easily deduce that, if a \fct\ $u$ is \rqc\ on $\omega$, then it is also \rqc\ on $\omega+i\rbb^n$. Then it follows directly from Theorem~\ref{rqc-qpsh-0} that $u$ is \qpsh\ on $\Omega$.

For the converse statement, consider a rigid \qpsh\ \fct\ $\psi$ on $\Omega:=\omega+i\rbb^n$. Pick a real affine subspace $\pi$ in $\rbb^n$ of dimension $q+1$, a ball $B \relc \pi \cap \omega$ and an affine linear \fct\ $l$ on $\pi$ such that $\psi \leq l$ on $\partial B$. After a complex linear change of coordinates of the form $z \mapsto \lambda z+p$, where $\lambda \in \rbb$ and $p \in \cbb^n$, we may assume that $\pi$ contains the origin and that $B=B^n_1(0)\cap \pi$. Given a positive number $R>0$, which will be specified later, and another ball $B_R:=B_R^n(0)\cap \pi$ in $\pi$, consider the set
$D_R:=B+iB_R$. Since $\Omega$ is a tube set, $B \relc \omega \cap \pi$ and since $0 \in \pi$, the set $D_R$ contains $B+i\{0\}^n$ and lies relatively compact in $\Omega \cap \pi^\cbb$, where $\pi^{\cbb}:=\pi+i\pi$. Moreover, the boundary of $D_R$ in $\pi^{\cbb}$ splits into two parts,
\[
A_1:= \partial B + i\cl{B_R} \qand A_2:= \cl{B}+ i(\partial B_R).
\]
Since $l$ is affine linear, $\psi$ is \qpsh\ on $\Omega$ and since $z \mapsto \|x\|^2_2-\|y\|^2_2=\sum_{j=1}^n \hbox{Re}(z_j^2)$ is \ph\ on $\cbb^n_z=\rbb^n_x+i\rbb^n_y$, it follows from Remark~\ref{prop-qpsh} that for every integer $k\in\nbb$ the \fct
\[
\psi_k(z):=\psi(x)-l(x)+\left(\|x\|^2_2-\|y\|^2_2 \right)/k
\]
is \qpsh\ on $\Omega$. The assumption $\psi\leq l$ on $\partial B$ and the choice of $D_R$ now yield the subsequent estimates for $\psi_k$ on the boundary of $D_R$,
\[
\psi_k \leq 1/k \quad \hbox{on}\ A_1 \qand \psi_k \leq \psi-l + (1-R^2)/k \quad \hbox{on}\ A_2.
\]
Thus, if we choose $R>0$ to be large enough, then $\psi_k$ becomes negative on $A_2$. Hence, the \fct\ $\psi_k$ is bounded by $1/k$ on the boundary of $D_R$. Since $\psi_k$ is \qpsh, the maximum principle implies that the \fct\ $\psi_k$ is bounded from above by $1/k$ on the closure of $D_R$ in $\pi^\cbb$. In particular, $\psi_k \leq 1/k$ on $B+i\{0\}^n$. But the last inequality holds for every integer $k \in \nbb$. This yields $\psi - l \leq 0$ on $B$, and we can conclude that $\psi$ is \rqc\ on $\omega$.
\eproof

As an application, we obtain a result for \qpsh\ \fcts\ on Reinhardt domains.

\begin{cor}\label{Koike-1} Let $V$ be an open set in $\rbb^n$ and consider the Reinhardt domain $\Omega_V:=\{z \in \cbb^n : (\ln|z_1|,\ldots,\ln|z_n|) \in V\}$. Then $u$ is \rqc\ on $V$ if and only if $\psi: z \mapsto u(\ln|z_1|,\ldots,\ln|z_n|)$ is \qpsh\ on $\Omega_V$.
\end{cor}

\begin{proof} Consider the holomorphic map $\Phi : V+i\rbb^n \to\Omega_V$ defined by $\Phi(w_1,\ldots,w_n)=(e^{w_1},\ldots,e^{w_n})=z$. Then $\psi(z)=(\psi \circ \Phi)(w_1,\ldots,w_n) = u(\repa(w_1),\ldots,\repa(w_2))$. Hence, the composition $\psi \circ \Phi$ is rigid on $V+i\rbb^n$. Now if $\psi$ is \qpsh\ on $\Omega_V$, the composition $\psi \circ \Phi$ is a rigid \qpsh\ \fct\ on $V+i\rbb^n$ according to Theorem~\ref{prop-qpsh}~(\ref{prop-qpsh-hol-change}). By Theorem~\ref{rqc-qpsh}, $u=\psi \circ \Phi$ is real $q$-convex on~$V$. Conversely, if $u$ is real $q$-convex as a function defined on $V$, $u$ is \qpsh\ on $V+i\rbb^n$ by Theorem~\ref{rqc-qpsh}. Since $\Phi$ is locally biholomorphic, we have that $\psi=u \circ \Phi^{-1}$ is (locally) \qpsh\ on $\Omega_V$. Then the rest of the proof follows from the identity $\psi(z)=(u \circ \Phi^{-1})(z)=u(\ln|z_1|,\ldots,|z_n|)$.
\end{proof}


\section{Real $q$-convex and $q$-pseudoconvex sets}\label{sect-rqc-sets}

We recall various notions of boundary distance \fcts\ and investigate their mutual relations.

\begin{defn} Let $\omega$ be an open set in $\rbb^n$ and let $\|\cdot\|$ be some arbitrary real norm on $\rbb^n$.

\begin{enumerate}

\item The \emph{boundary distance on $\omega$ induced by $\|\cdot\|$ } is given by
\[
\omega \ni x \mapsto d_{\|\cdot\|}(x,\partial\omega):=\inf\big\{\|x-y\|:y\in \partial\omega \big\}.
\]
We set $d_{\|\cdot\|}(x,\partial\omega):=+\infty$, if $\partial\omega$ is empty.

\item We write $d_2(x,\partial\omega):=d_{\|\cdot\|_2}(x,\partial\omega)$ for the boundary distance induced by the Euclidean norm~$\|\cdot\|_2$.

\item Let $v$ be a fixed vector in $\rbb^n$ with $\|v\|_2=1$ and let $x+ \rbb v$ be the real line in $\rbb^n$ that passes through $x$ and $x+v$. We define the \emph{(Euclidean) boundary distance in $v$-direction on $\omega$} by
\[
\omega \ni x \mapsto R_v(x,\partial\omega):=d_2\big(x , \partial\omega \cap ( x+ \rbb v )\big).
\]
\end{enumerate}
\end{defn}

We list the following elementary and well-known properties of these distance functions.

\begin{prop}\label{prop-dist} Let $\omega \subset\rbb^n$ be open, $x \in \omega$, $\|\cdot\|$ some real norm on $\rbb^n$. Then:

\begin{enumerate}

\item \label{distance4} $\D
d_{\|\cdot\|}(x,\partial\omega) = \inf\big\{ R_v(x,\partial\omega) \cdot \|v\| : v\in\rbb^n, \ \|v\|_2=1\big\}$.

\item \label{distance3} $
d_{v}(x,\partial\omega) =  d_{\|\cdot\|}(x,\partial\omega \cap ( x+ \rbb v ))/\|v\|$, where $v \in \cbb^n$ with $\|v\|_2=1$.

\item The boundary distance $x\mapsto d_{\|\cdot\|}(x,\partial \omega)$ is \cont\ on $\omega$.

\item For every vector $v \in \rbb^n$ with $\|v\|_2=1$ the boundary distance in $v$-direction $R_v$ is \lsc\ on $\omega$.

\end{enumerate}

\end{prop}

We will need the next property for our second main theorem.

\begin{lem}\label{rqc-v-dist-implies-rqc-dist} Let $\omega$ be an open set in $\rbb^n$ and $\|\cdot\|$ an arbitrary real norm on $\rbb^n$. Then:

\begin{enumerate}

\item If $x \mapsto -R_v(x,\partial \omega)$ is \rqc\ on $\omega$ for every vector $v\in\rbb^n$ with $\|v\|_2=1$, then $x \mapsto -d_{\|\cdot\|}(x,\partial\omega)$ is \rqc\ on $\omega$.

\item If $x \mapsto -\ln R_v(x,\partial \omega)$ is \rqc\ on $\omega$ for every vector $v\in\rbb^n$ with $\|v\|_2=1$, then $x \mapsto -\ln d_{\|\cdot\|}(x,\partial\omega)$ is also \rqc\ on $\omega$

\end{enumerate}

\end{lem}

\begin{proof}  By Proposition~\ref{prop-dist} and Proposition~\ref{prop-rqc}~(\ref{prop-rqc-sup}), we have that
\[
- d_{\|\cdot\|}(x,\partial\omega) =   \sup\{ -R_v(x,\partial\omega)\cdot \|v\| : v\in\rbb^n, \ \|v\|_2=1\big\},
\]
\[
\qand - \ln d_{\|\cdot\|}(x,\partial\omega) =   \sup\{ -\ln R_v(x,\partial\omega) - \ln \|v\| : v\in\rbb^n, \ \|v\|_2=1\big\},
\]
are both \rqc\ on~$\omega$ under the assumptions made in 1.~and 2., respectively.
\end{proof}

We can now deduce the real $(n-1)$-convexity of the negative of the distance functions.

\begin{prop} \label{dist-n-1-cvx}  Let $\omega \subset\rbb^n$ be open and let $\|\cdot\|$ be an arbitrary real norm on $\rbb^n$. Then the following four \fcts\ are all $(n-1)$-convex on~$\omega$:
\[
-R_v(x,\partial \omega), \quad -d_{\|\cdot\|}(x,\partial \omega), \quad -\ln R_v(x,\partial \omega) \qand -\ln d_{\|\cdot\|}(x,\partial \omega)
\]
\end{prop}

\begin{proof} Fix $p \in \omega$ and $v \in \rbb^n$ with $\|v\|_2=1$. Let $I_p$ be the connected component of $(p+\rbb v) \cap \omega$ containing $p$. Then $I_p$ is an open interval of the form $I_p=(a_p,b_p)$, where $a_p,b_p \in \rbb\cup\{\pm \infty\}$ and $a_p<b_p$. Moreover, for $x \in I_p=(a_p,b_p)$ we have $R_v(x,\partial \omega) = \min\{x-a_p,b_p-x\}$. But then $-R_v(x,\partial \omega) = \max\{a_p-x,x-b_p\}$ is convex for $x \in I_p$.

Now observe that, if $p,q \in \omega$, then either $I_p=I_q$ or, $I_p$ and $I_q$ are parallel to each other. 
Since $x \mapsto -R_v(x,\partial \omega)$ is locally convex on $I_p$ for every $p \in \omega$, Lemma~\ref{rqc-on-fol} implies that $x \mapsto -R_v(x,\partial \omega)$ is $(n-1)$-convex on~$\omega$. By a similar argument, the same is true for
\[
x \mapsto -\ln R_v(x,\partial \omega)=\max\{-\ln(x-a_p),-\ln(b_p-x)\}.
\]
Hence, by Lemma~\ref{rqc-v-dist-implies-rqc-dist}, both, $x\mapsto -d_{\|\cdot\|}(x,\partial \omega)$ and $x\mapsto -d_{\|\cdot\|}(x,\partial \omega)$, are $(n-1)$-convex on~$\omega$.
\end{proof}

We have seen in the proof that the one-dimensional case is special.

\begin{rem} Notice that in the case $n=1$, the \fcts\ $-d_{\|\cdot\|}(x,\partial\omega)$ and $-\ln d_{\|\cdot\|}(x,\partial\omega)$ are locally convex, i.e. real 0-convex, on \textbf{any} open set $\omega$ in $\rbb$ and \textbf{any} real norm~$\|\cdot\|$ on~$\rbb^n$. 
\end{rem}

Finally, real $q$-convexity is preserved under composition with strictly convex functions.

\begin{lem}\label{rqc-comp-sisc} Let $u$ be a \rqc\ on an open set $\omega$ in $\rbb^n$ and let $\varphi$ be strictly increasing and strictly convex. Then $\varphi \circ u$ is \rqc\ on $\omega$, as well.
\end{lem}

\begin{proof} Let $\Pi$ be a real $(q+1)$-dimensional subspace in $\rbb^n$, $B \Subset \omega$ a ball and $l:\rbb^n\to\rbb$ an affine linear function such that $\varphi \circ u \leq l$ on $\partial B \cap \Pi$. Since $\varphi^{-1}$ is also strictly increasing, we obtain $u \leq \varphi^{-1}\circ l$ on $\partial B \cap \Pi$. Since $\varphi^{-1}$ is strictly convex, $\varphi^{-1} \circ l$ is concave. But then by the definition of \rqcy, $u \leq \varphi^{-1}\circ l$ on $B \cap \Pi$. This yields $\varphi \circ u \leq l$ on $B \cap \Pi$. Thus, $\varphi \circ u$ is \rqc\ on $\omega$.
\end{proof}



Now we define generalized convex sets. 

\begin{defn}\label{defn-rqc-set} We say that an open set $\omega$ in $\rbb^n$ is \emph{\rqc} if $x \mapsto - \ln d_2(x,\partial \omega)$ is \rqc\ on~$\omega$.
\end{defn}

We obtain a complete characterization of real $(n-1)$-convex sets using Proposition~\ref{dist-n-1-cvx} together with Proposition~\ref{prop-dist}~(\ref{distance4}) applied to the Euclidean norm $\|\cdot\|_2$.

\begin{prop}\label{cor-qpsc-n-1}
\textbf{Any} open set $\omega$ in $\rbb^n$ is real $(n-1)$-convex.
\end{prop}

Another notion of generalized pesudoconvexity can be formulated by means of a continuity principle.

\begin{defn}\label{q-princ}

\begin{enumerate}

\item A set $A$ is called \emph{$m$-planar} if there exists an open set $U$ in $\rbb^n$ and a real $m$-dimensional affine subspace $\Pi$ such that $A = U \cap \Pi$. Its \emph{(relative) boundary} is given by $\partial A:=\partial U \cap \Pi$.


\item\label{q-princ-cont} An open set $\omega$ in $\rbb^n$ admits the \emph{$q$-continuity principle} if the following holds true: Let $\{A_t\}_{t\in[0,1]}$ be a family of $(q+1)$-planar sets in some open set $U$ in $\rbb^n$ that continuously depend on $t$ in the Hausdorff topology. Assume that the closure of $\bigcup_{t\in[0,1]}A_t$ is compact. If $\partial A_1$ \textbf{and} $A_t \cup \partial A_t$ lie in $\omega$ for each $t\in[0,1)$, then we already have that $A_1$ completely lies in $\omega$.

\end{enumerate}

\end{defn}

Geometric convexity alone is not sufficient to characterize real $q$-convex sets.

\begin{rem} Let $\rbb^*:=\rbb\setminus\{0\}$. Let us call an open set $\omega$ in $\rbb^n$ to be \emph{geometrically $q$-convex} if the following holds true: For every $(q+1)$-planar set $A$ with $\partial A \subset \omega$, we have $A \subset \omega$.Then it is clear that, if $\omega$ is geometrically $q$-convex, then it admits the $q$-continuity principle, since with the boundary $\partial A_1$ of a $(q+1)$-subspace $A_1$, also $A_1$ itself has to be in $\omega$. Anyhow, the converse is not true in general. Indeed, let $\omega = \rbb^*\times\rbb \subset \rbb^2$. Then $\omega$ possesses a real 0-convex (i.e., locally convex) exhaustion function
\[
u(x,y):=\max\{-d(x,\partial\rbb^*),-d(y,\partial\rbb)\},
\]
but $\omega$ is \textbf{not} convex, i.e., not geometrically 0-convex. Nevertheless, $\omega$ admits the 0-continuity principle (see Theorem~\ref{equivrqc} below). Moreover, the function $x\mapsto -\ln d(x,\partial \rbb^*)=-\ln|x|$ is locally convex on $\rbb^*$, i.e., real 0-convex. Thus, $\rbb^*$ is a real 0-convex set, but $\rbb^*$ is not convex, i.e., not geometrically 0-convex. 
\end{rem}

We now provide a list of equivalent characterizations of \rqc\ sets.

\begin{thm}\label{equivrqc}\label{prop-qpsc}
Let $q \in \{0\ldots,n-2\}$ and $\omega$ be an open set in $\rbb^n$. Then the following statements are all equivalent.
\begin{enumerate}


\item\label{equivrqc7} $\omega$ admits the $q$-continuity principle.

\item\label{equivrqc2} For every vector $v$ in $\rbb^n$ with $\|v\|_2=1$ the distance \fct\ in $v$-direction $x\mapsto-\ln R_v(x,\partial\omega)$ is \rqc\ on~$\omega$.

\item \label{equivrqc3*} For any real norm $\|\cdot\|$ the \fct\ $x \mapsto -\ln d_{\|\cdot\|}(x,\partial\omega)$ is \rqc\ on~$\omega$.

\item \label{equivrqc3} $\omega$ is \rqc, i.e., $x \mapsto -\ln d_2(x,\partial\omega)$ is \rqc\ on~$\omega$.

\item\label{equivrqc4} There exists a (not necessarily \cont)\ \rqc\ \fct\ $u$ on $\omega$ such that $\D\limsup_{x \to \partial\omega} u (x)=+\infty$.

\item\label{equivrqc5} $\omega$ admits a \cont\ \rqc\ exhaustion function $v$ on $\omega$, i.e., for every $c \in \rbb$ the set $\{x \in \omega : v(x) < c\}$ is relatively compact in $\omega$.






\end{enumerate}
\end{thm}

\begin{proof} Notice that if $\omega=\rbb^n$, then there is nothing to show. Hence, we assume from now on that $\omega$ is a proper subset of $\rbb^n$.

We shall prove the theorem by verifying  the following chain of implications:

\[\ba{ccccccccccccccccccccc}
\ref{equivrqc7}&\Rightarrow&\ref{equivrqc2}&\Rightarrow&\ref{equivrqc3*}&\Rightarrow&\ref{equivrqc3}&\Rightarrow&\ref{equivrqc5}& \Rightarrow & \ref{equivrqc4} & \Rightarrow & \ref{equivrqc7} & &&  & \\
\ea\]

\paragraph{\ref{equivrqc7}\mfolgt\ref{equivrqc2}} Assume that $u(x):=-\ln  R_v(x,\partial\omega)$ is not \rqc\ on $\omega$ for some fixed vector $v\in\rbb^n$ with $\|v\|_2=1$. Then there exists a real $(q{+}1)$-dimensional affine subspace~$\pi$ such that $u$ is not \rqc\ near a point $p$ in $\pi\cap{\omega}$. By Proposition~\ref{prop-rqc}~(\ref{prop-rqc-lin-change}), we can assume without loss of generality that $p=0$ and $\pi$ is equal to $\rbb^{q+1}{\times}\{0\}^{n-q-1}$. Let $\omega^*$ be an open subset in $\pi$ such that $\pi\cap{\omega}=\omega^*{\times}\{0\}^{n-q-1}$. Consider the \fct\
\[
\rho:\omega^*\to\rbb, \quad \rho(\xi):=-\ln R_v((\xi,0),\partial\omega).
\]
We claim that $v\notin\pi$. Otherwise, the vector $v$ can be written as $(w,0)$ for some $w\in\rbb^{q+1}$, so the function $\rho$ has the form
\[
\rho(\xi)=-\ln R_{w}(\xi,\partial \omega^*) \fe \xi \in \omega^*\subset\rbb^{q+1}.
\]
But then Proposition~\ref{dist-n-1-cvx} gives that $\rho$ is \rqc\ on $\omega^* \subset \rbb^{q+1}$, which contradicts the assumptions made on $\rho$ at the beginning of this step. Hence, from now on, we can assume that $v\notin\pi$.

Since $\rho$ is not \rqc\ near the origin in $\omega^*$, there exist a ball $B\relc{\omega^*}$ and an affine linear function $l:\rbb^{q+1}\to\rbb$ such that $\rho<l$ on $\partial B$, but $\rho(\xi_0) > l(\xi_0)$ at some $\xi_0\in B$. 

We move the graph of $l$ upwards and then downwards until the first contact with the graph of $\rho$ over a point $\xi_1 \in B$. Then we can assume that $\rho(\xi_1) = l(\xi_1)$,  $\rho \leq l$ on $\overline{B}$ and, especially, $\rho < l$ on $\partial B$. Observe that $-\ln(-(b-a)+1)+a \geq b$ for every $b < a+1$, and that we have equality if only if $b=a$. Then
\[
h(\xi):= -\ln\big(-(l(\xi)-l(\xi_1))+1\big) + l(\xi_1) \geq l(\xi)
\]
on $D:=\{\xi \in B : l(\xi) < l(\xi_1)+1\}$. Moreover, $h(\xi)=l(\xi)$ if and only if $\xi=\xi_1$. Clearly, we have $\xi_1 \in D$. It is now easy to see that $h>l\geq \rho$ on $\cl{D}\setminus \{\xi_1\}$ and $\rho(\xi_1)=l(\xi_1)=\rho(\xi_1)$.

Therefore, we have for $\xi \in D$ that
\[
g(\xi):=\big( -(l(\xi)-l(\xi_1))+1\big)\cdot e^{-l(\xi_1)} \leq e^{-l(\xi)} \leq e^{-\rho(\xi)}=R_v((\xi,0),\partial \omega)
\]
with equality if and only if $g(\xi)=g(\xi_1)$. Notice that $g$ is linear and $g(\xi)\geq 0$ for every $\xi \in \overline{D}$.

Observe next that, if $x=(\xi,0)\in{\omega}$, then for every real number $s \in (-\sigma,\sigma)$ the point $x+sv=(\xi,0)+sv$ lies in $\omega$ if and only if $0\leq \sigma < R_v(x,\partial\omega)$. Define for $t \in [0,1]$ the real $(q+1)$-planar sets 
\[
A_t:=\{(\xi,0)+tg(\xi)v: \xi \in D\}.
\]
Then $A_t \subset \omega$ for every $t \in [0,1)$, and $\partial A_t \subset \omega$ for every $t \in [0,1]$, but $A_1 \not\subset \omega$, since 
\[
(\xi_1,0)+g(\xi_1)v = (\xi_1,0) + R_v((\xi_1,0),\partial \omega)v \ \in A_1 \cap \partial \omega.
\]
Therefore, the family $\{A_t\}_{t \in [0,1]}$ violates the $q$-continuity principle. This is a contradiction to the assumption made on $x\mapsto-\ln R_v(x,\partial \omega)$ not being \rqc. Thus, we have shown the implication~\ref{equivrqc7}\mfolgt\ref{equivrqc2}.
\paragraph{\ref{equivrqc2}\mfolgt\ref{equivrqc3*}} This is a consequence of Lemma~\ref{rqc-v-dist-implies-rqc-dist}.
\paragraph{\ref{equivrqc3*}\mfolgt\ref{equivrqc3}} Simply take the Euclidean norm $\|\cdot\|:=\|\cdot\|_2$.
\paragraph{\ref{equivrqc3}\mfolgt\ref{equivrqc4}} The function $u(x)=-\ln d_2(x,\partial\omega)$ is \rqc\ on $\omega$ and admits the property that $u(x)$ tends to $+\infty$ whenever $x$ tends to $\partial\omega$.
\paragraph{\ref{equivrqc4}\mfolgt\ref{equivrqc5}} The function $v(x)=u(x) + \|x\|_2^2$ is a \cont\ \rqc\ exhaustion \fct\ for $\omega$.
\paragraph{\ref{equivrqc5}\mfolgt\ref{equivrqc7}} Assume that $\omega$ does not admit the $q$-continuity principle with the family $\{A_t\}_{t \in [0,1]}$ violating the corresponding properties, i.e., $A_t \subset \omega$ for all $t \in [0,1)$, $\partial A_t \in \omega$ for all $t \in [0,1]$, but $A_1 \not \subset \omega$. Set $K:=\bigcup_{t \in [0,1]} \partial A_t$ and let $v$ be an exhaustion function for $\omega$. By the maximum principle, we have for every $t \in [0,1)$ that
\[
\max_{A_t} v \leq \max_{\partial A_t} v \leq \max_{K} v =: C.
\]
Let $\{p_l\}_l$ be a sequence of points in $\omega$ such that $p_{l} \in \bigcup_{t \in [0,1)} A_{t}$ and $p_l \rightarrow p \in A_1 \cap \partial \omega$. Since $v$ is an exhaustion function, we have $\limsup_{l \to \infty} v(p_l) =+\infty$, but on the other hand, we concluded above that $\limsup_{l \to \infty} v(p_l) \leq C$. This contradiction means that our initial assumption on $\omega$ was wrong, so that in turn $\omega$ has to admit the $q$-continuity principle.
\end{proof}

As a direct application, we obtain further properties and examples of \rqc\ sets.

\begin{prop} \label{prop-qpsc-op} \label{prop-qpsc-op-int-2} \
\begin{enumerate}

\item If $u$ is \rqc\ on $\omega$, and $\omega$ is a \rqc\ set in $\rbb^n$, the the sublevel set $\omega_c:=\{x \in \omega: u(x)<c\}$ is a \rqc\ set for every $c \in \rbb$. 

\item\label{prop-qpsc-op-product} If $\omega_1$ is a \rqc\ in $\rbb^n$ and $\omega_2$ a \rqc\ set in $\rbb^m$, then $\omega_1 \times \omega_2$ is a \rqc\ set in $\rbb^{n+m}$.

\item \label{prop-qpsc-op-int} Let $\{\omega_j\}_{j \in J}$ be a collection of \rqc\ sets in $\rbb^n$ such that the interior $\omega$ of the intersection $\bigcap_{j \in J} \omega_j$ is not empty. Then $\omega$ is \rqc.


\end{enumerate}

\end{prop}

\begin{proof} 1.~We apply Lemma~\ref{rqc-comp-sisc} to  $\varphi(t)=-\ln(c-t)$ in order to obtain that $-\ln(c-u(x))$ is \rqc\ on $\omega_c$. Then $v(x):=\max\{-\ln(c-u(x))+\|x\|_2^2,-\ln d_2(x,\partial \omega)\}$ is the maximum of two \rqc\ \fcts\ and, therefore, \rqc\ on $\omega_c$ by itself. It is obvious that $v$ is a \rqc\ exhaustion function for $\omega_c$. 

\noindent 2.~For $j=1,2$, let $v_j$ be a \rqc\ exhaustion \fct\ of $\omega_j$. Then $v(x,y):=\max\{v_1(x),v_2(y)\}$ is a \rqc\ exhaustion \fct\ for the product set $\omega_1\times\omega_2$.

\noindent 3.~It is obvious that $d_2(x,\partial\omega)=\inf_{j \in J} d_2(x,\partial\omega_j)$ for every $x \in \omega$. Hence, $-\ln d(x,\partial\omega)$ is the supremum of the \rqc\ \fcts\ $-\ln d_2(x,\partial \omega_j)$ on $\omega$. Since it is also \cont\ on $\omega$, by Proposition~\ref{prop-rqc}~(\ref{prop-rqc-sup}), it is a \rqc\ exhaustion \fct\ for $\omega$.
\end{proof}

Having established the above results, we can prove the following interesting relation between affine linear maps and complements of \rqc\ sets. It can be regarded as the real analogue of Hartogs' theorem on the complement of holomorphic functions and its generalization to holomorphic maps~\cite{Oh,ShTP,Ma2}.



\begin{thm}\label{rqc-lin-fol-0} Let $f:\rbb^{n}\to\rbb^k$ be \cont. Then $f$ is affine linear on $\rbb^n$ if and only if the complement $\Gamma(f)^c$ of the graph $\Gamma(f)=\{(x,y): x \in \rbb^n, \ y=f(x)\}$ is a real $(k-1)$-convex set in $\rbb^{n+k}$.
\end{thm}

\begin{proof} 1.~If $f:\rbb^n\to\rbb^k$, $f(x)=y$, is affine linear, it is easy to verify that the real Hessian of 
\[
u(x,y)=-\ln\|f(x)-y\|_2 + \|(x,y) \|_2^2
\]
has at most $(k-1)$ negative eigenvalues at points $(x,y)$ with $f(x)\neq y$. By Theorem~\ref{smooth-rqc}, $u$ is strictly real $(k-1)$-convex and forms an exhaustion \fct\ for $\Gamma(f)^c$. Thus, by Theorem~\ref{equivrqc}, $\Gamma(f)^c$ is a real $(k-1)$-convex set in $\rbb^{n+k}$.

\noindent 2.~If $\Gamma(f)^c$ is a real $(k-1)$-convex set in $\rbb^{n+k}$, then $f_j$ is affine linear for each $j=1,\ldots,k$. If not, there is an index $j$ such that $f_j$ is not convex or $f_j$ or not concave. Without loss of generality, we can assume that $j=1$ and that $f_1$ is not convex. Then there are $x_{1},x_{2} \in  \rbb^n$ and $t_0 \in (-1,1)$ such that
\[
f_1\Big(\frac{1-t_0}{2}x_1 + \frac{1+t_0}{2}x_2\Big) > \frac{1-t_0}{2}f_1(x_1) + \frac{1+t_0}{2} f_1(x_2)=:y_0.
\]
Let $x_0:=\frac{1-t_0}{2}x_1 + \frac{1+t_0}{2}x_2$ and $r_0:=f_1(x_0)-y_0$. Consider the one-parameter family of real $k$-planar sets $\pi_r=\psi_r([-1,1]^k)$ defined as the trace of the parametrization $\psi_r:[-1,1]^k \to \rbb^{n+k}$ via
\begin{eqnarray*}
&&\psi_r(t,s_2,\ldots,s_k) \\
&&:= \Big( \frac{1-t}{2}x_1 + \frac{1+t}{2}x_2, \ \frac{1-t}{2}f_1(x_1) + \frac{1+t}{2}f_1(x_2)+r,\ f_2(x_0)+s_2, \ldots, f_k(x_0)+s_k\Big)
\end{eqnarray*}
where $t,s_2,\ldots,s_k \in [-1,1]$ and $r \geq r_0$. Then it is easy to verify that $\{\pi_r\}_{r_0\leq r\leq 2r_0}$ violates the $(k{-}1)$-continuity principle in Theorem~\ref{equivrqc}~(\ref{equivrqc7}) as $r\downarrow r_0$. Thus, $\Gamma(f)^c$ cannot be $(k-1)$-convex, a contradiction. Therefore, each $f_j$ has to be affine linear which means that $f=(f_1,f_2,\ldots,f_k)$ in total is affine linear on $\rbb^n$.
\end{proof}

We now compare \rqc\ sets to generalized pseudoconvex sets. The following definition is adapted from~\cite{Sl}. For $q=0$, it coincides with classical pseudoconvexity.

\begin{defn}\label{defn-qpsc} Let $q \in \{0,1\ldots,n-1\}$. An open set $\Omega$ in $\cbb^n$ is called \qpsc\ if $z \mapsto -\ln d_2(z,\partial \Omega)$ is \qpsh\ on $\Omega$.\footnote{The $q$-pseudoconvexity was originally introduced by Rothstein~\cite{Ro}. Another equivalent notion is the \emph{\pscy\ of order $n-q$} introduced by O. Fujita~\cite{Fu0}. In the smooth case, it is well-known as \emph{$q$-completeness} in the sense of Grauert.}
\end{defn}

Since we will only use the above definition here, for equivalent notions or lists of properties of \qpsc\ sets, we refer to \cite{Dieu}, \cite{Sl} or \cite{TPTHESIS}. There, one finds another characterizations of \qpsc\ sets, such as the following.

\begin{thm}\label{thm-local-qpsc} An open set $\Omega$ in $\cbb^n$ is \qpsc\ if and only if for each boundary point $p \in \partial \Omega$ there exists an open \nbh\ $U$ of $p$ and a \qpsh\ \fct\ $\psi$ on $\Omega \cap U$ such that $\psi(z) \to + \infty$ whenever $z \to \partial \Omega$ in $\Omega \cap U$.
\end{thm}

For the special case $q=n-1$, it is worth mentioning the following result from~\cite{Sl}. 

\begin{thm}\label{thm-n-1-psc} Every open set in $\cbb^n$ is $(n-1)$-pseudoconvex.
\end{thm}

Our next main result clarifies the relationship between \rqc\ and \qpsc\ sets.

\bthm[Second main theorem] \label{rqc-qpsc} Let $\omega$ be an open set in $\rbb^n$.

\begin{enumerate}

\item \label{rqc-qpsc-2}  If $\omega$ is a \rqc\ set in $\rbb^n$, then the set $\Omega=\omega+i(-a,a)^n$ is \qpsc\ for any $a \in (0,+\infty]$.

\item \label{rqc-qpsc-3}  If the set $\Omega=\omega+i(-a,a)^n$ is \qpsc\ for some $a \in (0,+\infty]$, then $\omega$ is a \rqc\ set in $\rbb^n$.

\end{enumerate}
\ethm

\bproof If $q= n-1$, there is nothing to show, since every open set in $\cbb^n$ is $(n-1)$-\psc\ according Theorem~\ref{thm-n-1-psc}, and every open set in $\rbb^n$ is $(n-1)$-convex due to our Corollary \ref{cor-qpsc-n-1}. Hence, from now on we assume that $q < n-1$. The convex/pseudoconvex case $q=0$ is due to Lelong~\cite{Le}.


\textbf{Case $\mathbf{a=+\infty}$.} In this case, we are in the setting of a tube set of the form  $\Omega=\omega+i\rbb^n$. Since $d_2(z,\partial\Omega)=d_2(\repa(z),\partial\omega)$ for every $z \in \Omega$, the \fct\  $z \mapsto d_2(z,\partial\Omega)$ is rigid on $\Omega$. Then it follows from Theorem~\ref{rqc-qpsh} that the \fct\ $x \mapsto -\ln d_2(x,\partial\omega)$ is \rqc\ on $\omega$ if and only if $z \mapsto -\ln d_2(z,\partial\Omega)$ is \qpsh\ on $\Omega$. Hence, $\omega$ is a \rqc\ set in $\rbb^n$ if and only if $\Omega$ is a \qpsc\ set in $\cbb^n$.

\medskip

\textbf{Case $\mathbf{a>0}$.} Assume that $\omega$ is \rqc. Then, in view of the previous case ($a=+\infty$), the set $\omega+i\rbb^n$ is \qpsc. Since the set $\rbb^n+i(-a,a)^n=\left(\rbb+i(-a,a)\right)^n$ is \psc\ as a product of \psc\ sets, it follows from Proposition~\ref{prop-qpsc-op} (\ref{prop-qpsc-op-int}), that the following intersection is \qpsc,
\[
(\rbb^n+i(-a,a)^n ) \cap (\omega+i\rbb^n) = \omega+i(-a,a)^n.
\]
In order to prove the converse direction, assume that $\Omega$ is \qpsc. Theorem~4.3.2 in~\cite{TPTHESIS} implies that, for every vector $v = u+ i0$ with $u \in \rbb^n$ and $\|u\|_2=1$, the function $-\ln R_v(z,\partial\Omega)$ is \qpsh\ on $\Omega$. Since $\Omega$ is of the form $\omega+i(-a,a)^n$, we have that
\[
(\rbb^n+i\{0\}^n) \cap \Omega = \omega + i\{0\}^n.
\]
But this means that $R_v(z,\partial \Omega)=R_u(\repa(z),\partial \omega)$. Hence, $-\ln R_v(z,\partial\Omega)$ is a well-defined rigid function on $\omega + i\rbb^n$. In view of Theorem~\ref{rqc-qpsh}, we obtain that for every $u$ with $\|u\|_2=1$ the \fct\ $-\ln R_u(x,\partial \omega)$ is \rqc\ on $\omega$. By Theorem~\ref{equivrqc}~(\ref{equivrqc2}), $\omega$ is a \rqc\ set in $\rbb^n$.
\eproof

We obtain a result similar to~Corollary~\ref{Koike-1}, but formulated for sets rather than for functions.

\begin{cor}\label{Koike-2} Let $V$ be an open set in $\rbb^n$ and let $\Omega_V:=\{z \in \mathbb{C}^n : (\ln|z_1|,\ldots,\ln|z_n|) \in V\}$. Then $V$ is \rqc\ in $\rbb^n$ if and only if $\Omega_V$ is \qpsc\ in $\cbb^n$.
\end{cor}

\begin{proof} By Theorem~\ref{rqc-qpsc} we know that $V$ is a \rqc\ set in $\rbb^n$ if and only if $V+i\rbb^n$ is \qpsc\ in $\cbb^n$. Consider the locally biholomorphic map $\Phi : V+i\rbb^n \to\Omega_V$ defined by $\Phi(w_1,\ldots,w_n)=(e^{w_1},\ldots,e^{w_n})=z$. It is obvious that $w \to \partial(V+i\rbb^n)$ in $V+i\rbb^n$ if and only if $\Phi(w)=z \to \partial \Omega_V$ in $\Omega_V$. Observe also that $\psi$ is \qpsh\ on an open subset $U$ of $\Omega_V$ if and only if $\psi\circ \Phi$ is \qpsh\ on the open subset $W:=\Phi^{-1}(U)$ in $V+i\rbb^n$, whenever $\Phi$ is biholomorphic on $W$. Then the result follows from Thoerem~\ref{thm-local-qpsc}.
\end{proof}

From this, we derive a generalized version of the classical fact on logarithmically convex Reinhardt domains (case $q=0$).

\begin{cor}\label{Koike-3} Let $D$ be a Reinhardt domain in $\cbb^n$. Then $D$ is \qpsc\ if and only if $\log D :=\{(\ln|z_1|,\ldots,\ln|z_n|) \in \rbb^n : z \in D\}$ is a \rqc\ set in $\rbb^n$.
\end{cor}

\begin{proof} Simply put $V:=\log D$. Then clearly $D=\Omega_V$, so that the statement follows directly from the previous corollary.
\end{proof}

\bibliographystyle{alpha}

\begin{thebibliography}{Tiefe}



\bibitem[BCP96]{BCP}
Bianchi,~G., Colesanti, A., and Pucci, C.:
\newblock On the second differentiability of convex surfaces.
\newblock Geom. Dedicata, 60(1), 39--48 (1996). \href{https://doi.org/10.1007/BF00150866}{https://doi.org/10.1007/BF00150866} 


\bibitem[Boc38]{Bo}
Bochner,~S.:
\newblock A theorem on analytic continuation of functions in several variables.
\newblock Ann. of Math. (2), 39(1), 14--19 (1938). \href{https://doi.org/10.2307/1968709}{https://doi.org/10.2307/1968709}

\bibitem[Bun90]{Bu}
Bungart,~L.: Piecewise smooth approximations to q-plurisubharmonic functions. Pacific J. Math., 142(2), 227--244 (1990). \href{http://projecteuclid.org/euclid.pjm/1102646343}{http://projecteuclid.org/euclid.pjm/1102646343}




\bibitem[Die06]{Dieu}
Nguyen Quang Dieu:
\newblock $q$-plurisubharmonicity and $q$-pseudoconvexity in $\cbb^n$.
\newblock Publ. Mat., Vol. 50, 349--369 (2006). \href{https://doi.org/10.5565/publmat\_50206\_05}{https://doi.org/10.5565/publmat\_50206\_05}

\bibitem[Fuj64]{Fu0}
Fujita, O.: Sur les familles d'ensembles analytiques. J. Math. Soc. Japan 16 (4) 379--405 (1964). \href{https://doi.org/10.2969/jmsj/01640379}{https://doi.org/10.2969/jmsj/01640379}

\bibitem[Fuj92]{Fu2}
Fujita, O.:
\newblock On the equivalence of the q-plurisubharmonic functions and the
  pseudoconvex functions of general order.
\newblock Ann. Reports of Graduate School of Human Culture, Nara Women's
  Univ. (7), 77--81 (1992).

\bibitem[HM78]{HM}
Hunt, L.~R., Murray, J.~J.: 
\newblock $q$-plurisubharmonic functions and a generalized Dirichlet problem, 
\newblock Michigan Math. J. \textbf{25} (3), 299--316 (1978). \href{https://doi.org/10.1307/mmj/1029002112}{https://doi.org/10.1307/mmj/1029002112}

\bibitem[Lel52]{Le}
Lelong, P.:
\newblock La convexit\'e et les fonctions analytiques de plusieurs variables
  complexes.
\newblock J. Math. Pures Appl., Serie 9, Vol. 31, 191--219 (1952). \href{https://www.numdam.org/item/JMPA_1952\_9\_31\_\_191\_0}{https://www.numdam.org/item/JMPA\_1952\_9\_31\_\_191\_0}


\bibitem[Mat23]{Ma2}
Matsumoto, K.: Hartogs' analyticity theorem for $C^2$-mappings and maximum principle for $q$-convex functions, Hiroshima Math. J. {\bf 53}, no.~3, 269--279 (2023). \href{https://doi.org/10.32917/h2022007}{https://doi.org/10.32917/h2022007}


\bibitem[PS22]{ShTP}
Pawlaschyk, T., Shcherbina, N. V.: Foliation of continuous q-pseudoconcave graphs, Indiana Univ. Math. J. 71 No. 4, 1627–1648 (2022). \href{https://doi.org/10.1512/iumj.2022.71.9010}{https://doi.org/10.1512/iumj.2022.71.9010}

\bibitem[Paw15]{TPTHESIS}
Pawlaschyk, T.:
\newblock On some classes of $q$-plurisubharmonic functions and $q$-pseudoconcave sets. Dissertation zur Erlangung eines Doktorgrades, 
\newblock Bergische Universit\"at Wuppertal (2015). \href{https://nbn-resolving.org/urn:nbn:de:hbz:468-20151210-101726-1}{https://nbn-resolving.org/urn:nbn:de:hbz:468-20151210-101726-1}

\bibitem[PZ13]{TPESZ}
Pawlaschyk, T., Zeron, E.~S.:
\newblock On convex hulls and pseudoconvex domains generated by $q$-plurisubharmonic functions, part I. 
\newblock J. Math. Anal. App. \textbf{408}, 394--408 (2013). \href{https://doi.org/10.1016/j.jmaa.2013.05.074}{https://doi.org/10.1016/j.jmaa.2013.05.074}


\bibitem[Ohs20]{Oh}
Ohsawa, T.:
\newblock Generalizations of theorems of Nishino and Hartogs by the $L^2$
  method.
\newblock Math. Res. Lett. 27, no. 6, 1867-–1884 (2020). \href{https://dx.doi.org/10.4310/MRL.2020.v27.n6.a12}{https://dx.doi.org/10.4310/MRL.2020.v27.n6.a12}

\bibitem[Roc70]{Rocka}
Rockafellar, R.~T.:
\newblock Convex analysis.
\newblock Princeton Mathematical Series, No. 28. Princeton University Press, Princeton, N.J. (1970). 

\bibitem[Rot55]{Ro}
Rothstein, W.:
\newblock Zur Theorie der analytischen Mannigfaltigkeiten im Raume von $n$ komplexen Ver\"anderlichen, 
\newblock Math. Ann. \textbf{129}, 96--138 (1955), \href{https://doi.org/10.1007/BF01362361}{https://doi.org/10.1007/BF01362361}


\bibitem[SSI25]{Sa}
Ismoilov, I., Sadullaev, A., Sharipov, R.: 
\newblock $m-cv$ measure $\omega^*(x,E,D)$ and condenser capacity $C(E,D)$ in the class $m$-convex functions,
\newblock J. Sib. Fed. Univ. Math. Phys., 18:3, 387–401 (2025)


\bibitem[S{\l}o84]{Sl2}
S{\l}odkowski, Z.:
\newblock The Bremermann-Dirichlet problem for $q$-plurisubharmonic
  functions.
\newblock Ann. Scuola Norm. Sup. Pisa Cl. Sci., Serie 4, Volume 11 no. 2, 303--326 (1984). \href{https://www.numdam.org/item/ASNSP\_1984\_4\_11\_2\_303\_0}{https://www.numdam.org/item/ASNSP\_1984\_4\_11\_2\_303\_0}

\bibitem[S{\l}o86]{Sl}
S{\l}odkowski, Z.:
\newblock Local maximum property and $q$-plurisubharmonic functions in uniform algebras, 
\newblock J. Math. Anal. Appl. \textbf{115} (1), 105--130 (1986). \href{https://doi.org/10.1016/0022-247X(86)90027-2}{https://doi.org/10.1016/0022-247X(86)90027-2}





\end{thebibliography}

\end{document}